\documentclass{amsart}
\usepackage{amsmath,amssymb,amsthm}
\usepackage{bm}
\usepackage{graphicx} % Required for inserting images
\usepackage{hyperref}

\def\IND{\setbox0=\hbox{$x$}\kern\wd0\hbox to 0pt{\hss$\mid$\hss}
	\lower.9\ht0\hbox to 0pt{\hss$\smile$\hss}\kern\wd0}
\def\NotIND{\setbox0=\hbox{$x$}\kern\wd0\hbox to 0pt{\mathchardef\nn=12854\hss$\nn$\kern1.4\wd0\hss}\hbox to 0pt{\hss$\mid$\hss}\lower.9\ht0 \hbox to 0pt{\hss$\smile$\hss}\kern\wd0}

\def\ind{\mathop{\mathpalette\IND{}}}
\def\nind{\mathop{\mathpalette\NotIND{}}}

\def\indi#1{\mathop{\mathpalette\IND{}^{\!\!\!\!\rlap{$\scriptstyle{#1}$}\,\,\,\,}}}

\def\nindi#1{\mathop{\mathpalette\NotIND{}^{\!\!\!\rlap{$\scriptstyle{#1}$}\,\,\,}}}

\theoremstyle{plain}
\newtheorem{thm}{Theorem}[section]
\newtheorem{lemma}[thm]{Lemma}
\newtheorem{fact}[thm]{Fact}
\newtheorem{prop}[thm]{Proposition}
\newtheorem{cor}[thm]{Corollary}
\newtheorem{maintheorem}{Theorem}

\theoremstyle{definition}
\newtheorem{defi}[thm]{Definition}
\newtheorem*{nota}{Notation}
\newtheorem*{exmp}{Example}

\theoremstyle{remark}
\newtheorem{rmk}[thm]{Remark}

\newcommand{\mR}[1]{\textup{MR}(#1)}

\newcommand{\MRP}[1]{\textup{MR}_P(#1)}
\newcommand{\SUP}[1]{\textup{SU}_P(#1)}
\newcommand{\td}[1]{\textup{trdeg}(#1)}

\newcommand{\A}{\mathbb{A}}
\newcommand{\B}[1]{B_P(#1)}
\newcommand{\eq}{^\textup{eq}}

\newcommand{\tp}[1]{\textup{tp}(#1)}
\newcommand{\tP}[1]{\textup{tp}_{L_P}(#1)}
\newcommand{\acl}[1]{\textup{acl}(#1)}
\newcommand{\dcl}[1]{\textup{dcl}(#1)}
\newcommand{\aclP}[1]{\textup{acl}_P(#1)}
\newcommand{\dclP}[1]{\textup{dcl}_P(#1)}
\newcommand{\aeq}[1]{\textup{acl}^\textup{eq}_P(#1)}
\newcommand{\deq}[1]{\textup{dcl}^\textup{eq}_P(#1)}
\newcommand{\Cb}[1]{\textup{Cb}(#1)}
\newcommand{\Pind}{P\text{-independent}}
\newcommand{\R}[1]{\textup{gR}(#1)}

\newcommand{\code}[1]{\lceil #1\rceil}

\title{Rank and Independence of Imaginaries in Proper Pairs of ACF}
\author{Zixuan Zhu}
\thanks{The author has been supported by the China Scholarship Council, the Deutsche Forschungsgemeinschaft (DFG, German Research Foundation) under the Excellence Strategy EXC 2044–390685587, Mathematics Münster: Dynamics–Geometry–Structure, and via the DFG project 495759320 Modelltheorie bewerteter Körper mit Endomorphismus.}
\address{Department for Mathematical Logic and Foundational Research, University of Münster}
\email{zixuan.zhu@uni-muenster.de}
\begin{document}
	\begin{abstract}
		Let $T_P$ be the theory of beautiful pairs of algebraically closed fields of fixed characteristic. It is known that for real tuples in models of $T_P$, SU-rank coincides with Morley rank and can be computed effectively.
		
		Building on Pillay’s geometric description \cite{PILLAY200713} of imaginaries in $T_P$, we define an additive rank on imaginaries of $T_P$, called the geometric rank. It takes values in $\omega\cdot\mathbb{N}+\mathbb{Z}$ and coincides with SU-rank on real tuples. It refines SU-rank and characterizes forking in $T_P\eq$. As a consequence, we derive an explicit criterion for determining forking independence.
	\end{abstract}
	
	\maketitle
	
	\section*{Introduction}
	The theory of proper pairs of algebraically closed fields (\textup{ACF}) has been extensively studied since the 1950s. It is known to be $\omega$-stable and to admit quantifier reduction. For real tuples, Morley rank and SU-rank coincide; moreover, they admit explicit descriptions, providing a satisfactory account of dimension and independence.
	
	However, this picture does not extend to imaginaries. Although SU-rank is still defined in $T_P^{\eq}$, it no longer preserves the same level of transparency and geometric control. For example, a generic element and its coset of the additive group of the small field both have rank $\omega$, while the generic element has rank 1 over the coset.
	
	In 2007, Pillay gave a geometric description of imaginaries in proper pairs of ACF \cite{PILLAY200713}: every imaginary is interalgebraic with one of a specific geometric form. We refer to imaginaries of this type as being in \emph{Pillay form}.
	
	In this article, we show that the group appearing in the Pillay form is canonical.
	\begin{maintheorem}\label{T:Pillay-alg}
		Let $\alpha=\code{G_{\B{\alpha}}(P)*a}$, $\beta=\code{H_{\B\beta}(P)*b}$ be of Pillay form. If $\beta$ is algebraic over $\alpha$, then the stabilizer of $\tP{ab/(ab)^c}$ is a definable homogeny from $G:=G_{\B\alpha}(P)$ to $H:=H_{\B\beta}(P)$. Moreover, if $\alpha$ and $\beta$ are interalgebraic, then this stabilizer is an isogeny.
	\end{maintheorem}
	Following Theorem A, we define a new ``rank'' on $(M,P)^{\eq}$ based on the canonical data of imaginaries, called the geometric rank and denoted by $\mathrm{gR}$. We show that this rank satisfies several natural properties (listed in Theorem~\ref{T:propertieslist} in the end of the paper), including a characterization of forking in $T_P\eq$:
	\begin{maintheorem}\label{T:nfequivrk}
		Let $e_1,e_2,e_3$ be finite tuples (possibly imaginaries). Then
		$$\R{e_1/e_2e_3}\leq\R{e_1/e_2},$$
		and
		$$\R{e_1/e_2e_3}=\R{e_1/e_2}\text{ if and only if }e_1\indi{P}_{e_2}e_3.$$
	\end{maintheorem}
	
	Moreover, we obtain a criterion \eqref{eq:conditions}, expressed in terms of canonical data from Pillay form, that determines forking among imaginaries.
	
	The paper is organized into three sections, in addition to the introduction.
	In the preliminaries, we introduce the theory $T_P$ of proper pairs of algebraically closed fields, fix notation, and recall known results on rank and independence in $T_P$.
	In the second section we introduce imaginaries of Pillay form and study their algebraicity, and we show that the group associated is canonical.
	In the final section, we define a geometric rank in terms of canonical data and show some natural properties. Indeed, we give an explicit condition, stated in \eqref{eq:conditions}, equivalent to non-forking in $T_P^{\eq}$ and to the decrease of the geometric rank. This condition reveals the underlying reasons for forking.
	
	\medskip
	\noindent
	\textbf{Acknowledgements.}
	The author would like to thank Martin Hils and Daniel Palac\'{i}n for raising the questions that led to this work, and Anand Pillay for valuable discussions.
	
	\section{Preliminaries}
	Let $L=L_{\text{ring}}$ be the language of rings, and let $T$ be a completion of the theory of algebraically closed fields (ACF) of fixed characteristic. It is well-known that $T$ is $\omega$-stable and admits quantifier elimination (QE) and elimination of imaginaries (EI). Moreover, the Morley rank and the SU-rank in $T$ coincide with transcendence degree, which is definable in families and additive.
	\subsection{Proper Pairs of ACF}
	To study the theory of proper pairs of models of $T$, we expand the language by adding a unary predicate symbol $P$ naming an elementary submodel. Denote by $L_P := L \sqcup \{P\}$ the expanded language of pairs.
	
	A \textbf{proper pair of ACF} is an $L_P$-structure $(M,P)$, such that $P(M)\subsetneq M$, where both $M$ and $P(M)$ are models of $T$. Since $T$ has QE, $P(M)$ is an elementary submodel of $M$.
	
	It is known that all such proper pairs are elementarily equivalent as $L_P$-structures. Denote the corresponding complete theory by $T_P$. In fact, $T_P$ coincides with the theory of \textit{beautiful pairs} of $T$, introduced by Poizat in \cite{pairstable} for stable theories and later generalized to \textit{lovely pairs} by Ben-Yaacov, Pillay and Vassiliev in \cite{lovelypairs} in the context of simple theories. 
	
	Every $\aleph_1$-saturated model $(M,P)\models T_P$ is a lovely pair, i.e., it satisfies the following properties for any countable parameter set $A\subset M$:
	\begin{enumerate}
		\item (\emph{Heir Property}) For any $L$-type $p \in S(A)$, there exists a non-forking extension of $p$ over $A \cup P(M)$ that is realized in $M$. 
		\item (\emph{Coheir Property}) For any $L$-type $p \in S(A)$, if $p$ does not fork over $A \cap P(M)$, then $p$ is realized in $P(M)$. 
	\end{enumerate}
	\subsection{Notation}
	We will frequently work with both languages $L$ and $L_P$. To distinguish the corresponding notions in the two languages, we use a $P$ superscript or subscript when needed. For instance:
	\begin{itemize}
		\item We denote by $\ind$ the non-forking independence in $T$, while $\indi P$ denotes non-forking independence in $T_P$.
		\item By $\acl a$ and $\aclP a$ we denote algebraic closure of $a$ in $L$ and in $L_P$, respectively.
	\end{itemize}
	
	Other model-theoretic notions (such as types, definable closure, etc.) are distinguished in the same way.
	
	For $A\subset(M,P)\models T_P$ a parameter set, we set $P(A):=A\cap P(M)$. We also write $P$ for $P(M)$. We say $A$ is \textbf{$P$-independent} if $A\ind_{P(A)}P$.
	
	For a tuple $a\in M^n$, we set $a^c:=\Cb{a/P}$. Since $T$ is $\omega$-stable and admits EI and $P$ is purely stably embedded, we can consider $a^c$ as a finite real tuple in $P$. Then $a^c\in \dclP{a}$ and the tuple $(a,a^c)$ is $\Pind$.
	\subsection{Properties for $T_P$}
	The following fact shows that, for a $P$-independent set $A$, the $L_P$-structure of $A$ is determined by its $L$-structure.
	\begin{fact}[{\cite[3.8]{lovelypairs}}, {\cite[2.5]{Imgnrinbtfprs}} and \cite{pairstable}]\label{F: lovelypairs}
		Let $(M,P)\models T_P$. Let $A,B\subset(M,P)$ be subsets that are $P$-independent. Then, we have:
		\begin{enumerate}
			\item If $\text{qftp}_{L_P}(A)=\text{qftp}_{L_P}(B)$, then $\tP A=\tP B$.
			\item The algebraic closure in the sense of $T_P$ coincides with that in $T$, i.e., $\aclP{A}=\acl{A}$.
			\item The theory $T_P$ is $\omega$-stable.
			\item If $C$ is a common $\Pind$ subset of $A$ and $B$, then:
			$$A \indi{P}_C B\quad\text{ if and only if }\quad 
			\left\{
			\begin{aligned}
				&A \ind_C B\quad\text{and},\\
				&P(A) \ind_{P(C)} P(B).
			\end{aligned}
			\right.$$
			\item $P(M)$ is purely stably embedded in $M$.
		\end{enumerate}		
	\end{fact}
	The last statement follows from quantifier reduction in $T_P$ and the fact that $T$ is stable.
	\subsection{Ranks of Real Tuples}
	Poizat \cite{Poizat2001LgalitAC} described the ranks of real tuples in models of $T_P$ (see, e.g., \cite{martinpizarro2024noetheriantheories} for a proof).
	\begin{fact} \label{F:Amador}
		Let $C\preceq M\models T_P$ and $a$ be a real tuple from $M$. Then 
		\begin{equation*}
			\MRP {a/C}=\SUP {a/C}=\omega\cdot\td{a/CP}+\td{(aC)^c/C}.
		\end{equation*}
	\end{fact}
	
	In particular, Morley rank and SU-rank coincide for types of real tuples. From the above formula, one can show these ranks are additive in the sense of Cantor addition, see \cite{Poizat2001LgalitAC}.

	\section{Algebraicity between Imaginaries of Pillay Form}
	In this section, we will introduce a class of imaginaries in $T_P$ of a specific form, which we call ``Pillay form'', referring to Pillay's theorem on a geometric description of imaginaries in $T_P$. Then, we will study algebraicity between imaginaries of Pillay form.
	\begin{defi}
		Let $G$ be a connected algebraic group and $V$ be an absolutely irreducible variety, both defined over some algebraically closed field $k$. A rational function $\mu:G\times V\dashrightarrow V$ defined over $k$ is called a \textbf{generically free rational group action} if for any $K\supset k$, any $v\in V(K)$ generic over $k$, and any $g_1,g_2\in G(k)$, we have $(g_1,v),(g_2,\mu(g_1,v))\in \text{dom} (\mu)$, $\mu(g_1,\mu(g_2,v))=\mu(g_1*g_2,v)$ and $\mu(g_1,v)=v$ if and only if $g_1=\text{id}$.
	\end{defi}
	\begin{defi}
		Let $(M,P)\models T_P$. Let $B\subset\A^n, V\subset \A^m$ be irreducible varieties over the prime field together with a surjective morphism $\pi:V\rightarrow B$ such that for any $b\in B$ generic over the prime field, $V_b:=\pi^{-1}(b)$ is an absolutely irreducible variety defined canonically over $b$. Let $G\rightarrow B$ be an algebraic group over $B$, i.e., every fiber $G_b$ is an algebraic group and the group operations are defined uniformly over $b$.
		
		Suppose moreover there exists a $\emptyset$-definable subset $\mu\subset G\times_B V\times_B V$ such that for every $b\in B(P)$, the fiber $\mu \cap G_b\times V_b\times V_b$ is the graph of a partial map $\mu_b: G_b\times V_b\dashrightarrow V_b$, and that for generic $b\in B(P)$, $\mu_b$ is a generically free rational group action.
		
		Let $\phi(x,y)$ be the $L_P$-formula expressing the following: $y\in B(P),x\in V_y$ and for any $g_1,g_2\in G_y(P)$,
		$$(g_1,x),(g_2,\mu_y(g_1,x)) \in \text{dom} (\mu_y),$$  $$\mu_y(g_1,\mu_y(g_2,x))=\mu_y(g_1*g_2,x)\in V_y,$$
		$$\mu_y(g_1,x)=x\text{ if and only if }g_1=\text{id}_{G_y}.$$
		Set
		$$W:=\{x\in V\mid M\models\phi(x,\pi(x))\}.$$
		
		We define a binary relation $\sim$ on $W$ by setting $w_1 \sim w_2$ if $\pi(w_1)=\pi(w_2)$ and there exists $g\in G_{\pi(w_1)}(P)$ such that $\mu (g,w_1)=w_2$. One can check that $\sim$ defines an equivalence relation on $W$. We denote the quotient sort by
		$$\mathcal{S}(B,V,G,\mu):=W/\sim.$$
		
		An imaginary $e\in(M,P)\eq$ is said to be of \textbf{Pillay form} if it is a generic element in some Pillay sort $\mathcal{S}(B,V,G,\mu)$. That is, there exists a generic $b$ in $B(P)$ and $a\in V_b(M)$ a generic over $P$, such that $e=\lceil G_b(P)*a\rceil$\footnote{Let $X$ be a definable set. We denote by $\code{X}$ the canonical parameter of $X$.}.
	\end{defi}
	\begin{rmk}
		In general, let $T$ be a stable theory with EI and nfcp. Let $T_P$ be the theory of beautiful pairs of models of $T$. Likewise, we can define Pillay sorts: Let $G$ be a definable group acting definably on some definable set $V$ in the sense of $L$. Assume that the induced group action of $G(P)$ on $V$ is generically free, and that all the elements in an $L_P$-generic orbit have the same $L_P$-type over $P$. Then the canonical parameter of such an orbit is said to be of Pillay form.
	\end{rmk}
	\begin{fact}[\cite{PILLAY200713}]\label{F:Pillay-gei}
		Let $(M,P)\models T_P$. For any imaginary $e\in(M,P)\eq$, there exists an imaginary $e'$ of Pillay form such that $e$ and $e'$ are interalgebraic.
	\end{fact}
	\begin{nota}
		Given $e\in(M,P)\eq$, denote $\B e:=\aeq{e}\cap P(M)$.
	\end{nota}
	\begin{rmk}       
		In the original statement of Fact~\ref{F:Pillay-gei}, it is not required that $b$ be the canonical parameter of $V_b$. 
		However, it follows from the proof in \cite{PILLAY200713} that $b$ may in fact be chosen canonically. 
		In this case, one obtains $\acl{b}=\B{e}$.
		
		Moreover, G.~J.~Boxall shows in his PhD thesis \cite{BoxallThesis} that $(M,P)$ admits weak elimination of imaginaries with imaginaries of Pillay form. 
		More precisely, for every imaginary $e\in(M,P)\eq$, there exists an imaginary $e'$ of Pillay form such that 
		$e\in\deq{e'}$ and $e'\in\aeq{e}$.
	\end{rmk}	
	If two imaginaries of Pillay form are interalgebraic, then there are some connections between the groups associated to the imaginaries. In fact, we have the following result.
	\begin{thm}[Theorem~\ref{T:Pillay-alg}]
		Let $\alpha=\code{G_{\B{\alpha}}(P)*a}$, $\beta=\code{H_{\B\beta}(P)*b}$ be of Pillay form. If $\beta$ is algebraic over $\alpha$, then the stabilizer of $\tP{ab/(ab)^c}$ is a definable homogeny from $G:=G_{\B\alpha}(P)$ to $H:=H_{\B\beta}(P)$. Moreover, if $\alpha$ and $\beta$ are interalgebraic, then this stabilizer is an isogeny.
	\end{thm}
	
	In stable theories, homogenies and isogenies are defined as follows:
	\begin{defi}
		Let $G$, $H$ be two (type-)definable groups. 
		\begin{itemize}
			\item A \textbf{homogeny} from $G$ to $H$ is a (type-)definable subgroup $S\subset G\times H$, such that the projection of $S$ to the first coordinate $\pi_1(S)$ is a subgroup of bounded index in $G$ and the cokernel coker$(S)=\{h\in H|\,(1,h)\in S\}$ is finite.
			
			If moreover $\pi_2(S)$ is a subgroup of bounded index in $H$ and the subgroup ker$(S)=\{g\in G|\,(g,1)\in S\}$ is finite as well, $S$ is called an \textbf{isogeny}.
			
			\item If there is an isogeny between $G$ and $H$, then we say $G$ and $H$ are \textbf{isogenous}.
		\end{itemize}
	\end{defi}
	In our context, since $T_P$ is $\omega$-stable, we can replace bounded index by finite index in the definition, and every type-definable group is definable.
	\begin{rmk}
		Let $G$, $H$ be two groups. Let $S$ be a subgroup of $G\times H$. Then $S$ induces an isomorphism: $$\frac{\pi_1(S)}{\text{ker}(S)}\overset{\sim}{\longrightarrow}\frac{\pi_2(S)}{\text{coker}(S)}.$$ So if $G$ and $H$ are isogenous, they have the same SU-rank.
	\end{rmk}
	\begin{proof}[Proof of Theorem \ref{T:Pillay-alg}]
		We may assume that $(M,P)$ is a lovely pair. Suppose $\alpha=\code{G_{\B\alpha}(P)*a}$ and $\beta=\code{H_{\B\beta}(P)*b}\in (M,P)\eq$ both of Pillay form such that $\beta$ is algebraic over $\alpha$.
		
		Let $p=\tP{ab/(ab)^c}$. For convenience, let $X$ and $Y$ denote the orbits $G_{\B\alpha}(P)*a$ and $H_{\B\beta}(P)*b$, respectively. Let $G$ and $H$ denote the groups $G_{\B\alpha}(P)$ and $H_{\B\beta}(P)$, respectively. The group $G\times H$ acts on $X\times Y$ naturally, and we define $S$ as the stabilizer of $p$, i.e., $$S:=\{(g_1,g_2)\in G\times H\mid(g_1*a,g_2*b)\equiv^P_{(ab)^c} ab\}.$$ 
		Since $T_P$ is $\omega$-stable, $S$ is an $(ab)^c$-definable subgroup of $G\times H$. We will show that $S$ is a homogeny.
		
		First, we show that the projection of $S$ to $G$, denoted by $\pi_1(S)$, has finite index in $G$. (In fact, we will see, since $G$ is connected, that $\pi_1(S)=G$). Suppose for contradiction that $\pi_1(S)$ has infinite index in $G$. Then there exists $\{g_i\}_{i\in\omega}\subset G$ with $g_0=1$ such that $\pi_1(S)g_i\cap \pi_1(S)g_j=\emptyset$ for $i\neq j$. Since $G_{\B{\alpha}}$ is connected in the sense of $T$, $G$ is connected in the sense of $T_P$ as $G\subset P$ and $P$ is purely stably embedded in $M$. Hence $g_i\in G\subset\text{Stab}(\tP{a/\B e})$. Denote $a_i:=g_i*a\in X$. From the definition of Pillay form, we know $a_i\equiv_{P(M)}a$ and hence $a_i\equiv_{(ab)^c}a$. Then there exists $\{b_i\}_i$ such that $a_ib_i\equiv_{(ab)^c} ab$ and $b_0=b$.
		
		Notice that $a_i=(g_ig_j^{-1})*a_j$. If for some $i\neq j$ there exists $h\in H$ such that $b_i=h*b_j$, then $(g_ig_j^{-1},h)*(a_j,b_j)=(a_i,b_i)$. As $a_ib_i\equiv_{(ab)^c}a_jb_j$ we know that $(g_ig_j^{-1},h)\in S$, contradicting $g_ig_j^{-1}\notin \pi_1(S)$. Therefore, $b_i$ and $b_j$ are in different orbits under the group action of $H$ for $i\neq j$. That is, if we denote by $\beta_i$ the canonical parameter of the orbit of $b_i$, we have $\beta_i\neq\beta_j$ for $i\neq j$. However, $a_ib_i\equiv^Pa_jb_j$ since $a_ib_i\equiv_{(ab)^c}a_jb_j$. So after taking quotients, we still have $\alpha \beta_i\equiv^P\alpha\beta_j$, i.e., all the orbits $\{\beta_i\}_{i\in\omega}$ have the same type over $\alpha$. This contradicts that $\beta=\beta_0$ is algebraic over $\alpha$. Therefore, $\pi_1(S)$ has finite index in $G$.
		
		On the other hand, to show coker$(S)$ is finite, we claim that $\tp{b/aP}$ is algebraic. Using the heir property of lovely pairs, we may take a Morley sequence $\{b_i\}_{i\in\omega}$ in $\tp{b/aP}$ in $M$, where $b_0=b$.
		Since $\beta$ is $L_P$-algebraic over $\alpha$, there are at most finitely many orbits in Loc$(b/P)$ having the same $L_P$-type as $\beta$ over $\alpha$. Hence, there exists an integer $k$ such that for any $i\in\omega$, there exists $0\leq k_i<k$ such that $b_i$ and $b_{k_i}$ are in the same orbit.
		
		Note that $\{b_{i+k}\}_{i\in\omega}$ is a Morley sequence in $\tp{b_k/b_0\dots b_{k-1}aP}$. Moreover, for every $i\geq k$, $b_i\in\dcl{b_0\dots b_{k-1}aP}$ as there exists $h\in H\subset P$ such that $b_i=h*b_{k_i}$. Hence the type $\tp{b_i/b_0\dots b_{k-1}aP}$ is algebraic, thus so is $\tp{b/aP}$.
		
		Now, since $ab\ind_{(ab)^c}P$, $b\ind_{a(ab)^c}P$. It follows that $b$ is algebraic over $a(ab)^c$, and there are only finitely many $b'$ such that $ab\equiv_{(ab)^c} ab'$. As $H$ acts freely on $Y$, there are thus only finitely many $h\in H$ satisfying that $a,b\equiv_{(ab)^c}a,h*b$. That is to say, coker$(S)$ is finite.
		
		In conclusion, $S=\text{Stab}(\tp{ab/(ab)^c})$ is a homogeny from $G$ to $H$. In case $\alpha$ and $\beta$ are interalgebraic, by the same arguments $S$ is an isogeny.
	\end{proof}
	\begin{rmk}
		In the proof of Theorem \ref{T:Pillay-alg}, we only use the stability of \textup{ACF}. In fact, as long as one has an analogue of imaginaries of Pillay form in a general setting, Theorem~\ref{T:Pillay-alg} can be generalized. 
	\end{rmk}
	\begin{exmp}
		Let $(M,P)\models T_P$. Let $e_1$ be a generic coset in $(M,+)/(P(M),+)$ and $e_2$ be a generic coset in $M^\times/P^\times$, then $e_1$ and $e_2$ are not interalgebraic as $(P,+)$ and $(P^\times,\times)$ are not isogenous.
		
		However, Theorem~\ref{T:Pillay-alg} does not hold in the theory $\text{DCF}_0$ of differentially closed fields of characteristic 0 (without predicate). Let $(K,\partial)\models\text{DCF}_0$ and $C$ be the constant field and $a\in K$ be a generic point. Let $b,c\in K$ such that $\partial(b)=a$ and $\frac{\partial(c)}{c}=a$. Then $b+C=\{x\in K\mid \partial(x)=a\}$ and $c\cdot C^\times=\{x\in K\mid \frac{\partial(x)}{x}=a\}$. So $\code{b+C}$ and $\code{c\cdot C^\times}$ are interdefinable, but $(C,+)$ and $(C^\times,\cdot)$ are not isogenous.  
	\end{exmp}
	\section{Geometric Rank and Non-forking Independence}
	In this section, based on the analysis of interalgebraicity between imaginaries of Pillay form, we define a geometric rank on all imaginaries in $T_P$ and show some properties of the geometric rank. In particular, we show that geometric rank witnesses non-forking independence.
	\begin{defi}
		Let $(M,P)\models T_P$. Let $e=\lceil G_b(P)*a\rceil\in(M,P)\eq$ be of Pillay form. We define the \textbf{geometric rank} of $e$ to be: 
		$$\R{e}:=(\td{a/P(M)},\,\td{b}-\dim{G_b}).$$
	\end{defi}
	
	Notice that $\td{a/P(M)}$ is independent of the choice of $a$ as different choices are interdefinable over $P$. Moreover, the definition of Pillay form yields $\td{b}=\td{\B e}$ since $b$ and $\B{e}$ are interalgebraic.
	
	The geometric rank $\R{e}$ takes values in $(\mathbb{Z}\times_{\text{lex}}\mathbb{Z})^{\geq 0}$, the non-negative part of $\mathbb{Z}\times_{\text{lex}}\mathbb{Z}$. In the following, we write $\omega\cdot n+z$ for $(n,z)\in\mathbb{Z}\times_{\text{lex}}\mathbb{Z}$. For example, let $e$ be a generic coset in $(M,+)/(P(M),+)$. Then $\R{e}=\omega-1$.
	
	In fact, the geometric rank is invariant under interalgebraicity. 
	\begin{thm}\label{T:welldefined}
		Let $\alpha, \beta$ be imaginaries of Pillay form such that $\beta$ is $L_P$-algebraic over $\alpha$. Then, $\R{\alpha}\geq\R{\beta}$.
		In particular, if $\alpha$ and $\beta$ are interalgebraic, then $\R{\alpha}=\R{\beta}$.
	\end{thm}
	\begin{proof}
		Let $\alpha=\code{G_{\B\alpha}(P)*a}$ and $\beta=\code{H_{\B\beta}(P)*b}$. In fact, in the proof of Theorem~\ref{T:Pillay-alg}, we already showed that $b$ is $L$-algebraic over $aP(M)$. Hence, we have $\td{b/P(M)}\leq\td{a/P(M)}$. If $\td{a/P(M)}>\td{b/P(M)}$, then it follows that $\R{\alpha}>\R{\beta}$ and we are done. In the rest of the proof, we assume that $\td{a/P(M)}=\td{b/P(M)}$.
		
		Since $\aeq{\beta}\subset\aeq{\alpha}$, $\B{\beta}\subset \B{\alpha}$. So, it follows that $\td{\B{\alpha}}\geq\td{\B{\beta}}$.
		
		Now we look at the group dimension part. As $b$ is $L$-algebraic over $aP(M)$ and we assumed that $\td{a/P(M)}=\td{b/P(M)}$, one can deduce that $a$ is also algebraic over $bP(M)$. Then symmetrically, by the same arguments as in the proof of Theorem~\ref{T:Pillay-alg}, the kernel of $S$ is also finite.
		
		As the homogeny $S$ induces an isomorphism $\pi_1(S)/\text{ker}(S)\xrightarrow{\sim}\pi_2(S)/\text{coker}(S)$ and $\pi_1(S)$ has finite index in $G$, we have: $$\dim{G}=\dim{\pi_1(S)/\text{ker}(S)}=\dim{\pi_2(S)/\text{coker}(S)}=\dim{\pi_2(S)}\leq\dim{H}.$$
		Therefore, $\td{\B\alpha}-\dim{G}\geq\td{\B{\beta}}-\dim{H}$, i.e., $\R{\alpha}\geq\R{\beta}$.
	\end{proof}
	\begin{rmk}
		With the notation from Theorem~\ref{T:welldefined}, if $\alpha$ and $\beta$ are interalgebraic, we conclude not only that $\R{\alpha}=\R{\beta}$ but also that $a$ and $b$ are $L$-interalgebraic over $P$, $\B{\alpha}$ and $\B{\beta}$ are interalgebraic and $G$ and $H$ are isogenous.
	\end{rmk}
	\begin{defi}
		Let $(M,P)\models T_P$ and $e,e'\in(M,P)\eq$.
		\begin{itemize}
			\item By Fact~\ref{F:Pillay-gei}, there exists $\alpha\in(M,P)\eq$ of Pillay form, such that $e$ and $\alpha$ are interalgebraic. We define the geometric rank of $e$ to be:
			$$\R{e}:=\R{\alpha}.$$
			By Theorem \ref{T:welldefined}, this is well-defined.
			\item Now, for any $e,e'\in(M,P)\eq$, we regard the tuple $ee'$ as an imaginary element. We define the geometric rank of $e$ over $e'$ to be: 
			$$\R{e/e'}:=\R{ee'}-\R{e'}.$$
		\end{itemize}
	\end{defi}
	By Theorem \ref{T:welldefined}, since $e'$ is algebraic over $ee'$, we have $\R{e'}\leq\R{ee'}$. Hence $\R{e/e'}\geq0$, thus still takes value in $(\mathbb{Z}\times_{\text{lex}}\mathbb{Z})^{\geq 0}$.
	
	From the definition, it is clear that geometric rank is invariant under automorphism. In the rest of this section, our goal is to show that it witnesses non-forking independence. 
	\begin{defi}
		Let $X\subset M^n$ be an $L_P$-definable set and $G\subset P(M)^m$ be an algebraic group defined over $\code{X}$. We call $X$ a \textbf{$G$-torsor} if $G$ acts freely and transitively on $X$, the group action is given by a rational function that can be defined over $a^c$ for any $a\in X$, and the group action is regular on $X$. Moreover, if for any $a,a'\in X$, $\tP{a/\code{X}P}=\tP{a'/\code{X}P}$, we say that $X$ is a \textbf{homogeneous $G$-torsor}.
	\end{defi}
	In fact, an imaginary is the code of a homogeneous torsor of a connected group if and only if it is of Pillay form.
	\begin{prop}
		Let $e=\code{G_b(P)*a}$ be an imaginary of Pillay form. Then for any $a'\in G_b(P)*a$, the tuples $a$ and $a'$ have the same $L_P$-type over $eP$. In particular, $G_b(P)*a$ is a homogeneous $G_b(P)$-torsor. Conversely, for any connected group $G$, the code of any homogeneous $G$-torsor is an imaginary of Pillay form.
	\end{prop}
	\begin{proof}
		From the definition of Pillay form, Loc$(a/P)=$Loc$(a'/P)=V_b$. Hence $\tp{a/P}=\tp{a'/P}$, so $\tP{a/P}=\tP{a'/P}$. Since $G_b$ and the group action are defined over $b\in\deq{a}$, it follows that $e\in\deq{a}$. Hence $\tP{a/eP}=\tP{a'/eP}$. 
		
		Conversely, let $X$ be a homogeneous $G$-torsor and fix $a\in X$. Then all elements in $X$ have the same $L$-type over $P$. Let $V_b=$Loc$(a/P)$ be the irreducible variety defined over $b=a^c= P\cap \deq{\code{X}}$. Since $X$ is a homogeneous $G$-torsor and the group $G$ is defined over $b$ and acts on $X$ by rational functions, $G$ also acts on the generics of $V_b$. Hence the action is a generically free rational group action. Thus $\code{G*a}$ is of Pillay form.
	\end{proof}
	\begin{rmk}\label{R: gRcalculation}
		Let $X$ be a homogeneous $G$-torsor (,where $G$ is not necessarily connected).  Then
		$$\R{\code{X}}=\omega\cdot\dim{\overline X^P}\!\!+\td{\B{\code{X}}}-\dim{G},$$
		where $\overline X^P\!\!$ denotes the smallest Zariski closed subset defined over $P$ containing $X$. Indeed, for any $a\in X$, $\dim{\overline X^P}\!\!=\td{a/P}$.
		
		The formula is immediate when $G$ is connected. In general, there exists some $G^0$-torsor that is interalgebraic with $\code{X}$, hence the same formula applies.
	\end{rmk}
	\begin{lemma}\label{L:hmg}
		Let $X$ be a $G$-torsor for some definable group $G$ in $P$. Suppose $Y\subset X$ is a definable subset that isolates a type over $\code{Y}P$ such that the group action is definable over $\code{Y}$. Then there exists a definable subgroup $H\subset G$ such that $Y$ is a homogeneous $H$-torsor.
	\end{lemma}
	\begin{proof}
		Take
		$$H:=\{g\in G\mid \forall y\in Y,\,g*y\in Y\}.$$
		From the definition, $H$ is closed under group multiplication and can be defined over $P$ as $P$ is stably embedded. For any $y\in Y$, $g\in H$, we have $g*y\in Y$. As $g*y$ and $y$ have the same type over $g\code{Y}$, $y$ also satisfies that there exists $y'\in Y$ such that $g*y'=y$, i.e., $g^{-1}*y\in Y$ for all $y\in Y$. We conclude $H$ is closed under inversion, hence a subgroup of $G$.
		
		As $G$ acts freely on $X$, $H$ also acts freely on $Y$. For any $y,y'\in Y$, since $G$ acts transitively on $X$, there exists $g\in G$ such that $g*y=y'$. Then, for any $y''\in Y$, as $y$ and $y''$ have the same type over $g\code{Y}$, $g*y''\in Y$. Thus, $g\in H$, that is, $H$ acts transitively on $Y$.
	\end{proof}
	Since both non-forking independence and the geometric rank are invariant under interalgebraicity, it is enough to consider imaginaries of Pillay form in what follows. Take $e_i=\code{X_i}$, where $X_i$ is a homogeneous $G_i$-torsor for some $G_i=G^i_{b_i}(P)$ for $i=1,2,3$, where the $G_{b_i}$'s are connected algebraic groups.
	
	\begin{lemma}[see also {\cite[2.5 and 2.6]{PILLAY200713}}]\label{L:e12}
		There exists a definable subgroup $G_{12}\subset G_1\times G_2$, and a homogeneous $G_{12}$-torsor $X_{12}\subset X_1\times X_2$ such that $e_{12}:=\code{X_{12}}$ is interalgebraic with the tuple $e_1e_2$.
	\end{lemma}
	\begin{proof}
		Let $B:=\aeq{e_1e_2}$, $P(B):=B\cap P(M)=\B{e_1e_2}$. Since $T_P\eq$ is $\omega$-stable, there exists $(a_1,a_2)\in X_1\times X_2$ such that $p:=\tP{a_1a_2/B}$ is isolated. Since $(a_1a_2)^c\in\dclP{a_1a_2}\cap P$ and $P$ is stably embedded, $\tp{(a_1a_2)^c/P(B)}$ is also isolated. However, as $P(B)$ is a model of $T$, all the isolated types are realized. Hence, $(a_1a_2)^c\in B$. Therefore, $a_1a_2\indi{P}_{e_1e_2}P$.
		
		Let $X_{12}$ denote the definable set corresponding to the isolating formula, we have $\code{X_{12}}\in B$. As $X_1\times X_2$ is defined over $B$ and contains $a_1a_2$, $X_{12}$ is a subset of $X_1\times X_2$. Notice that $p$ is stationary and $a_1a_2\indi{P}_{B}P$, $X_{12}$ also isolates a type over $BP$. Apply Lemma~\ref{L:hmg} to $X_{12}\subset X_1\times X_2$, where $X_1\times X_2$ is a $G_1\times G_2$-torsor. We find:
		$$G_{12}:=\{g\in G_1\times G_2\mid \forall x\in X_{12},\,g*x\in X_{12}\},$$
		such that $X_{12}$ is a homogeneous $G_{12}$-torsor. Moreover, since $\tP{x/B}$ is stationary, $G_{12}$ is connected.
		
		As $X_{12}$ is defined over $B=\aeq{e_1e_2}$, $e_{12}:=\code{X_{12}}\in\aeq{e_1e_2}$. On the other hand, for any $(a_1,a_2)\in X_{12}$, $e_1\in\deq{a_1},\,e_2\in\deq{a_2}$ uniformly. Hence $e_1e_2\in\deq{e_{12}}$.
	\end{proof}
	
	\begin{rmk}
		From the proof, $e_{12}$ depends on the choice of the elements of minimal rank in $X_1\times X_2$. In fact, there can be infinitely many $e_{12}$ satisfying the condition in Lemma~\ref{L:e12}, and they are all interalgebraic. In the following, we fix any of these to be $e_{12},\,G_{12}$ and $X_{12}$.
	\end{rmk}
	
	In the rest of this section, we will study $\R{e_1/e_2}-\R{e_1/e_2e_3}$ and whether $e_1\indi{P}_{e_2}e_3$. In either case, we may substitute $e_1$ with $e_{12}$ and $e_3$ with $e_{23}$, as they are interalgebraic over $e_2$. On the other hand, take any $a_i\in X_i$ for $i=1,2,3$, and if $a_1\nind_{a_2P}a_3$, then it is easy to see that $\R{e_1/e_2}-\R{e_1/e_2e_3}>0$ and $e_1\nindi{P}_{e_2}e_3$ both hold. We will therefore focus on the case that $a_1\ind_{a_2P}a_3$.
	
	Let $H_1\subset G_2$ denote the image of the projection $G_{12}\hookrightarrow G_1\times G_2\twoheadrightarrow G_2$, and let $H_3\subset G_2$ denote the corresponding image of $G_{23}$. For any $(a_1,a_2)\in X_{12},(a_2',a_3)\in X_{23}$, there is a unique $g\in G_2$ such that $g*a_2=a_2'$ as $a_2,a_2'\in X_2$ and $X_2$ is a $G_2$-torsor. This determines a unique double coset $K:=H_3gH_1$ in $G_2$ which is independent of the choice of $a_2,a_2'$. Indeed,
	$$K=\{g\in G_2\mid \exists xyy'z\,(x,y)\in X_{12}\wedge(y',z)\in X_{23}\wedge g*y=y'\}.$$
	We can see that $K$ is definable over $e_{12}e_{23}$. Since $K$ is a (real) non-empty definable set in $P(M)$, which is purely stably embedded in $M$, and $P(M)\models T$ admits EI, we conclude that $K$ is definable over $\B{e_1e_2e_3}$.
	
	For the same reason, the code $\code{K}$ is a real tuple. As $\acl{\code{K}}$ is a model of $T$, $K(\acl{\code{K}})$ is non-empty. We fix $g_0\in K$ such that $g_0\in\acl{\code{K}}$.
	
	\begin{lemma}\label{L:e123}
		Let $g_0$ be taken as above. Let
		$$G_{123}:=\{(g_1,g_0^{-1}g_2g_0,g_2,g_3)\mid(g_1,g_0^{-1}g_2g_0)\in G_{12},(g_2,g_3)\in G_{23}\};$$
		$$X_{123}:=\{(x,y,y',z)\mid(x,y)\in X_{12},\,(y',z)\in X_{23},\,g_0*y=y'\}.$$
		Then $X_{123}$ is a $G_{123}$-torsor and is a finite union of homogeneous torsors for the connected component of $G_{123}$. Moreover, the code of each torsor is interalgebraic with the tuple $e_1e_2e_3$.
	\end{lemma}
	\begin{proof}
		It is easy to check that $G_{123}$ is a subgroup of $G_{12}\times G_{23}$, and $X_{123}$ is a $G_{123}$-torsor in a natural way and is defined over $\aeq{e_1e_2e_3}$.
		
		Take any $(a_1,a_2,a_2',a_3)\in X_{123}$ where $a_2'=g_0*a_2$. Let $b_{12}$ denote $\B{\code{X_{12}}}$ and $b_{23}$ denote $\B{\code{X_{23}}}$. In the proof of Lemma~\ref{L:e12}, we showed $a_1a_2\ind^{P}_{e_1e_2}P$. As $\Cb{\tp{a_1a_2/P}}=\Cb{\tP{a_1a_2/P}}$, it follows that $a_1a_2\ind_{b_{12}}P$. By monotonicity, $a_1a_2\ind_{g_0a_2b_{12}}a_2P$. As we assumed $a_1\ind_{a_2P}a_3$, we have $a_1a_2\ind_{a_2P}a_2'a_3P$. By transitivity, $a_1a_2\ind_{g_0a_2b_{12}}a_2'a_3P$. Thus, $a_1a_2\ind_{g_0b_{12}b_{23}a_2'a_3}P$. On the other hand, it follows symmetrically that $a_2'a_3\ind_{b_{23}}P$, hence $a_2'a_3\ind_{g_0b_{12}b_{23}}P$. We conclude that $a_1a_2a_2'a_3\ind_{g_0b_{12}b_{23}}P$.
		
		Let $\bar{a}$ denote the tuple $(a_1,a_2,a_2',a_3)$. Let $p:=\tP{\bar{a}/\aeq{\code{X_{123}}}}$. We will show that $p$ is isolated. Let $\bar B$ denote $B_P(\code{X_{123}})$. Since $\bar{a}\ind_{g_0b_{12}b_{23}}P$ and $g_0b_{12}b_{23}\in\aeq{\code{X_{123}}}$, we conclude that $\aeq{\bar{a}}\cap P=\bar B=\acl{g_0b_{12}b_{23}}$. Let $(M_0,\bar B)\preceq (M_1,\bar B)$ be the prime model over $\code{X_{123}}$ and over $\bar a$ respectively. Take any $\bar a_0\in X_{123}(M_0)$. In $(M_1, \bar B)$, there exists $\bar g\in G_{123}$ such that $\bar{a}\in\dcl{\bar g,\bar a_0}$. However, $\bar g\in \bar B\subseteq M_0$, so we have $\bar a\in M_0$. That is, $p$ is realized in the prime model over $\aeq{\code{X_{123}}}$. Therefore, $p$ is isolated. Moreover, from $\bar{a}\ind_{g_0b_{12}b_{23}}P$, we know $\bar{a}\indi{P}_{g_0b_{12}b_{23}}P$, hence $\bar{a}\indi{P}_{\code{X_{123}}}P$.
		
		Let $Y$ denote the set of realizations of $p$, $Y$ is a definable subset of $X_{123}$. Since $p$ is stationary, $p$ admits a unique non-forking extension to $\aeq{\code{X_{123}}}P$. That is, for any $\bar{a},\bar{a}'\in Y$, $\tP{\bar{a}/\aeq{\code{X_{123}}P}}=\tP{\bar{a}'/\aeq{\code{X_{123}}P}}$. Applying Lemma~\ref{L:hmg} to $Y$, we get $Y$ is a homogeneous torsor.
		
		Let $Y'$ be the set of realizations of some other $p'=\tP{\bar{a}'/\aeq{\code{X_{123}}}}$ where $\bar{a}'\in X_{123}$. Similarly, we have $p'$ is isolated and $\bar{a}'\indi{P}_{\code{X_{123}}}P$. Then $Y$ and $Y'$ are the realizations of the extensions of $p$ and $p'$ to $\aeq{\code{X_{123}}P}$. Notice that there exists $g\in G_{123}\subset P$ such that $g*\bar{a}=\bar{a}'$, which is also a translation of $Y$ to $Y'$. We conclude $Y$ and $Y'$ have the same Morley rank. Moreover, since all the types in $\langle X_{123}\rangle\subset S(\aeq{\code{X_{123}}})$ are isolated, there are only finitely many types in $X_{123}$, i.e., $\code{Y}\in\aeq{\code{X_{123}}}$ and $\MRP{Y}=\MRP{X_{123}}=\dim G_{123}$. In fact, $Y$ is a homogeneous $G^0_{123}$-torsor.
		
		At last, we show $\code{Y}$ is interalgebraic with the tuple $e_1e_2e_3$. As we mentioned before, $K$ is defined over $e_{12}e_{23}$, and $g_0\in\acl{\code{K}}$, we know $g_0\in\aeq{e_1e_2e_3}$. From the definition $X_{123}$ is defined over $e_{12}e_{23}g_0$. Hence, $\code{Y}\in\aeq{\code{X_{123}}}\subset\aeq{e_1e_2e_3}$. On the other hand, for any $(a_1,a_2,a_2',a_3)\in Y$, $e_{12}\in\deq{a_1a_2}$ and $e_{23}\in\deq{a_2'a_3}$ uniformly. Therefore, $e_1e_2e_3\in\deq{\code{Y}}$.
	\end{proof}
	In the proof of Lemma~\ref{L:e123}, we showed that for any $\bar{a}\in Y$, $\bar{a}\ind_{g_0b_{12}b_{23}}P$. Hence $\B{\code{Y}}=\acl{g_0b_{12}b_{23}}$. By Remark~\ref{R: gRcalculation}, the geometric rank $\R{e_1e_2e_3}=\R{\code{Y}}$ is given by:
	$$\R{e_1e_2e_3}=\omega\cdot\dim{\overline Y^P}\!\!+\td{g_0b_{12}b_{23}}-\dim{G_{123}}.$$
	
	From the definition of $G_{123}$, it is easy to see that:
	\begin{equation}\label{eq:G1230}
		\dim G_{123}=\dim G_{12}-\dim H_1+\dim G_{23}-\dim H_3+\dim (g_0H_1g_0^{-1}\cap H_3)
	\end{equation}
	Notice that there is a natural bijection between cosets:
	$$g_0H_1g_0^{-1}/(g_0H_1g_0^{-1}\cap H_3)\overset{\sim}{\rightarrow}g_0H_1g_0^{-1}H_3/H_3$$
	We know 
	$$\dim{g_0H_1g_0^{-1}\cap H_3}=\dim{H_1}+\dim{H_3}-\dim{g_0H_1g_0^{-1}H_3}.$$
	As $g_0H_1g_0^{-1}H_3$ is definably bijective to $K=H_3g_0H_1$, $\dim{g_0H_1g_0^{-1}H_3}=\dim{K}$. Combined with equation \eqref{eq:G1230}, we conclude:
	\begin{equation}\label{eq:G123}
		\dim G_{123}=\dim G_{12}+\dim G_{23}-\dim{K}
	\end{equation}
	
	Now, let's calculate the value of $\R{e_1/e_2}-\R{e_1/e_2e_3}$. Recall that we assumed $a_1\ind_{a_2P}a_3$. Hence the $\omega$ part of the geometric rank automatically vanishes. For convenience, we denote $b_i:=\B{e_i}$ and denote $b_{ij}:=\B{e_ie_j}$ as before. We have:
	\begin{equation}
		\begin{split}
			\R{e_1/e_2}-\R{e_1/e_2e_3}&\\
			=\R{e_{12}}+\R{e_{23}}-&\R{e_{123}}-\R{e_2}\\
			=\,\td{b_{12}}-\dim& G_{12}+\td {b_{23}}-\dim G_{23}\\
			-\td{g_0b_{12}b_{23}&}+\dim G_{123}-\td {b_2}+\dim G_2\\
			\overset{\eqref{eq:G123}}{=}\dim G_2-&\dim{K}-(\td {g_0b_{12}b_{23}}-\td {b_{12}b_{23}})+J,
		\end{split}\label{eq:R1/2-R1/23}
	\end{equation}
	where 
	\begin{equation*}
		\begin{split}
			J:=&\td {b_{12}}+\td {b_{23}}-\td {b_2}-\td {b_{12}b_{23}}\\
			=&\td{b_{12}/b_2}-\td{b_{12}/b_{23}}.
		\end{split}
	\end{equation*}
	Since $b_2\in\acl{b_{23}}$, $J\geq0$ and the equality holds iff $b_{12}\ind_{b_2}b_{23}$.
	Take $g\in K$ of maximal transcendence degree over $g_0b_{12}b_{23}$. Then,
	$$\dim{K}=\td{g/g_0b_{12}b_{23}}=\td{gg_0b_{12}b_{23}}-\td{g_0b_{12}b_{23}}.$$
	As $H_1,H_3$ are defined over $b_{12}b_{23}$, and $K=H_3gH_1$, $\code{K}\in\dcl{gb_{12}b_{23}}$. In particular, $g_0\in\acl{\code{K}}\subset\acl{gb_{12}b_{23}}$. Hence, $\dim{K}=\td{gb_{12}b_{23}}-\td{g_0b_{12}b_{23}}$. We obtain: 
	$$\dim G_2-\dim{K}-(\td {b_{12}b_{23}g_0}-\td {b_{12}b_{23}})=\dim G_2-\td{g/b_{12}b_{23}}.$$
	Since $g\in G_2$, $\dim G_2-\td{g/b_{12}b_{23}}\geq\td{g}-\td{g/b_{12}b_{23}}\geq 0$. Substituting into \eqref{eq:R1/2-R1/23}, we conclude:
	\begin{prop}\label{P:gRequivstar}
		Let $e_i=\code{X_i}$ be imaginaries of Pillay form where $X_i$ is a homogeneous $G_i$-torus for $i=1,2,3$. Let $b_i$, $b_{ij}$ and $K$ be defined as above. Take any $a_i\in X_i$. Then 
		$$\R{e_1/e_2e_3}\leq\R{e_1/e_2}.$$
		Moreover, equality holds if and only if the following three conditions are satisfied:
		\begin{equation}
			\tag{$\bigstar$}
			\left\{\begin{aligned}
				\text{(a)}\;& a_1\ind_{a_2P}a_3;\\
				\text{(b)}\;& b_{12}\ind_{b_2}b_{23};\\
				\text{(c)}\;& \text{There exists } g\in K\text{ such that }g\text{ is generic in }G_2 \text{ over }b_{12}b_{23}.
			\end{aligned}\right.
			\label{eq:conditions}
		\end{equation}
	\end{prop}
	The conditions \eqref{eq:conditions} are analogous to Fact~\ref{F: lovelypairs}(4), which characterizes non-forking independence in $T_P$ for real tuples. We now show that \eqref{eq:conditions} also characterizes non-forking independence in $T_P^{\eq}$.
	\begin{prop}
		Let $e_1,e_2,e_3$ be of Pillay form. Under the same notation as above, $e_1\indi{P}_{e_2}e_3$ if and only if \eqref{eq:conditions} holds.
	\end{prop}
	\begin{proof}   
		$(\Leftarrow)$ Assume \eqref{eq:conditions} holds. Let $g\in K$ be a generic element of $G_2$ over $b_{12}b_{23}$. We fix $a_1,a_2,a_2',a_3\in M$ such that $(a_1,a_2)\in X_{12},\,(a_2',a_3)\in X_{23}$ and $g*a_2=a_2'$.
		
		As $\tP{a_2/e_2}$ is isolated, $\MRP{a_2/e_2}=\dim{G_2}$. Since $a_2'\ind_{b_2}P$, by Fact~\ref{F: lovelypairs}(4) $a_2'\indi{P}_{b_2}P$. Thus $\MRP{g/a_2'}=\MRP{g/b_2}=\dim G_2$ as $g\in P$ and as $g$ is generic in $G_2$ over $b_2$. Since $\tP{a_2/e_2}$ is isolated and its realizations form a $G_2$-torsor, we conclude $\MRP{a_2/e_2}=\dim{G_2}=\MRP{g/a_2'}=\MRP{a_2/a_2'}$. The last equation follows from the fact that $g$ and $a_2$ are interdefinable over $a_2'$. Therefore, $a_2\indi{P}_{e_2}a_2'$.
		
		The rest is done by forking calculus. Since $g\ind_{b_{2}}b_{12}b_{23}$, it follows from (b) that $b_{12}\ind_{b_2}gb_{23}$. Combined with (a), using Fact~\ref{F: lovelypairs}(4), we get $a_1\indi{P}_{a_2}a_2'a_3$. 
		
		On the other hand, since $g\ind_{b_2}b_{23}$, again by Fact~\ref{F: lovelypairs}(4), $a_2\indi{P}_{a_2'}a_3$. As we showed $a_2\indi{P}_{e_2}a_2'$, it follows that $a_2\indi{P}_{e_2}a_2'a_3$. Notice that $a_1\indi{P}_{a_2}a_2'a_3$, we conclude $a_1a_2\indi{P}_{e_2}a_2'a_3$.
		
		As $e_{12}\in\deq{a_1a_2}$ and $e_{23}\in \deq{a_2'a_3}$, we have $e_1\indi{P}_{e_2}e_3$. 
		
		$(\Rightarrow)$ Now, assume $e_1\indi{P}_{e_2}e_3$, i.e., $e_{12}\indi{P}_{e_2}e_{23}$. By the extension property, there exists $(a_1,a_2)\in X_{12}$ such that $a_1a_2\indi{P}_{e_{12}}e_{23}$. So $a_1a_2\indi{P}_{e_2}e_{23}$. Again, we choose $(a_2',a_3)\in X_{23}$ such that $a_2'a_3\indi{P}_{e_{23}}a_1a_2$, then $a_1a_2\indi{P}_{e_2}a_2'a_3$.
		
		Let $g\in K$ be the unique element such that $g*a_2=a_2'$.
		From $a_1a_2\indi{P}_{e_2}a_2'a_3$, we know that $a_1\indi{P}_{a_2}a_2'a_3$. Hence $a_1\ind_{a_2P}a_3$ and $b_{12}\ind_{b_2}gb_{23}$. Thus, both (a) and (b) hold and $g\ind_{b_{23}}b_{12}$.
		
		From $a_1a_2\indi{P}_{e_2}a_2'a_3$, we also have $a_2\indi{P}_{e_2}a_2'$ and $a_2\indi{P}_{a_2'}a_3$. From $a_2\indi{P}_{e_2}a_2'$ one knows that $\MRP{a_2/e_2}=\MRP{a_2/a_2'}$. Hence $\dim G_2=\MRP{a_2/e_2}=\MRP{a_2/a_2'}=\MRP{g/a_2'}$. Notice that $\MRP{g/a_2'}=\mR{g/b_2}$ as $\acl{b_2}=\aclP{a_2'}\cap P$. We have $\dim G_2=\mR{g/b_2}$, i.e., $g$ is generic in $G_2$ over $b_2$. The latter independence $a_2\indi{P}_{a_2'}a_3$ yields $g\ind_{b_2}b_{23}$. Combined with $g\ind_{b_{23}}b_{12}$, we have $g\ind_{b_2}b_{12}b_{23}$. Thus, $g$ is generic in $G_2$ over $b_{12}b_{23}$. Condition (c) is satisfied.
	\end{proof}
	\begin{exmp}
		Let $a\in M\setminus P(M)$ and $b\in P$ generic. Define 
		\[
		e_1 = \code{a + P}, \quad e_2 = \code{a \cdot P^\times}, \quad \text{and} \quad e_3 = \code{a \cdot b + P}.
		\]
		One can check that the tuple $e_1e_2$ is interdefinable with the real element $a$, the tuple $e_2e_3$ is interdefinable with the element $a\cdot b$ and the tuple $e_1e_2e_3$ is interdefinable with the tuple $(a,b)$.
		
		Hence we have
		\[
		\R{e_1 / e_2} = \omega - (\omega - 1) = 1, \quad 
		\R{e_1 / e_2 e_3} = (\omega + 1) - \omega = 1,
		\]
		so that $\R{e_1/e_2e_3}=\R{e_1/e_2}$.
		
		On the other hand, $a$ and $a\cdot b$ are independent in $a\cdot P^\times$ as $b$ is a generic in $P^\times$. (This corresponds to the condition (c) in \eqref{eq:conditions}). Thus, $a\indi{P}_{e_2}a\cdot b$ and hence $e_1\indi{P}_{e_2}e_3$.
		
		Note that the free 1-dimensional contribution of $e_1$ over $e_2$ comes from the group part, while the free 1-dimensional contribution of $e_1$ over $e_2 e_3$ comes from the parameter part ($\B{e}$ part in the definition). These two contributions interact with each other.
	\end{exmp}
	
	In conclusion, for $e_i$ imaginaries of Pillay form, $\R{e_1/e_2e_3}=\R{e_1/e_2}$ and $e_1\indi{P}_{e_2}e_3$ are equivalent, as they are both equivalent to \eqref{eq:conditions}. Moreover, since geometric rank and non-forking independence are invariant under interalgebraicity, and every imaginary is interalgebraic with an imaginary of Pillay form, we obtain:
	\begin{thm}[Theorem~\ref{T:nfequivrk}]
		Let $(M,P)\models T_P$. Let $e_1,e_2,e_3$ be finite tuples in $(M,P)\eq$. Then
		$$\R{e_1/e_2e_3}\leq\R{e_1/e_2},$$
		and
		$$\R{e_1/e_2e_3}=\R{e_1/e_2}\text{ if and only if }e_1\indi{P}_{e_2}e_3.$$
	\end{thm}
	\begin{defi}
		Let $e\in (M,P)\eq$ be an imaginary and $C\subset(M,P)\eq$ (possibly infinite). We define $$\R{e/C}:=\R{e/C_0},$$ 
		for some finite $C_0\subset\aeq{C}\text{ such that } e\indi{P}_{C_0} C$.
	\end{defi}
	Since $T_P$ is $\omega$-stable, there exists finite $C_0\subset\aeq{C}$ such that $e\indi{P}_{C_0}C$. Assume that $C_0'\subset \aeq{C}$ is a finite subset over which $\tP{e/C}$ also does not fork. Then by Theorem~\ref{T:nfequivrk}, $\R{e/C_0}=\R{e/C_0C_0'}=\R{e/C_0}$. Hence $\R{e/C}$ is well-defined. On the other hand $\R{e/C}=\min\{\R{e/C_0}\mid C_0\overset{\text{finite}}{\subset}C\}$.
	\begin{cor}\label{C:gRequivfrk}
		Let $(M,P)\models T_P$. Let $e\in(M,P)\eq$ and $B\subset C\subset(M,P)\eq$. Then $\R{e/B}=\R{e/C}$ if and only if $e\indi{P}_{B}C$.
	\end{cor}
	\begin{proof}
		Take $B_0\subset \aeq{B}$ finite such that $e\indi{P}_{B_0}B$, then $\R{e/B}=\R{e/B_0}$ by definition. From the analysis above, we know $\R{e/C}=\R{e/B_0}$ if and only if $e\indi{P}_{B_0}C$ thus if and only if $e\indi{P}_BC$.
	\end{proof}
	By comparison with Fact~\ref{F:Amador}, we see that for real tuples $a$ and real parameter sets $C$, the geometric rank $\R{a/C}=\SUP{a/C}$. In fact, we have the following corollary. 
	\begin{cor}\label{C:gRa=SUa}
		Let $C\subset(M,P)\eq$ and $a\in M$ be a real tuple. Then $\R{a/C}=\SUP{a/C}$.
	\end{cor}
	\begin{proof}
		Let $(M_0,P(M_0))\preceq(M,P)$ be a model containing $C$. We take $a'\equiv_C^P a$ such that $a'\indi{P}_{C} M_0$. Then 
		$$\R{a/C}=\R{a'/C}=\R{a'/M_0}=\SUP{a'/M_0}=\SUP{a/C}.$$
	\end{proof}
	Corollary~\ref{C:gRa=SUa} shows that the geometric rank coincides with SU-rank on real tuples. 
	We now collect the main properties established in this section, which together show that the geometric rank behaves as a well-behaved independence rank on $(M,P)^{\eq}$.
	\begin{thm}\label{T:propertieslist}
		Let $a,b\in (M,P)^{\eq}$ be finite tuples and $B,C\subset (M,P)^{\eq}$. Then,
		\begin{itemize}
			\item If $aB\equiv^P bC$, then $\R{a/B}=\R{b/C}$.\hfill(Aut-invariance)
			\item $\R{a/B}=\R{a/\aeq{B}}$.\hfill (Base closure)
			\item If $\aeq{aB}=\aeq{bB}$, then $\R{a/B}=\R{b/B}$.\hfill(Alg-invariance)
			\item $\R{ab/B}=\R{a/bB}+\R{b/B}$.\hfill(Additivity)
			\item If $B\subset C$, then $\R{a/C}\leq\R{a/B}$.\hfill(Monotonicity)
			\item $\R{a/B}=0$ if and only if $a\in\aeq{B}$.\hfill (Anti-reflexivity)
			\item $a\indi{P}_BC$ if and only if $\R{a/B}=\R{a/C}$.\hfill (Forking-characterization)
			\item $\R{a/C}=\min\{\R{a/C_0}\mid C_0\subset\aeq{C},\, C_0\text{ finite}\}.$\hfill(Local character)
			\item If $a\in M$, $\R{a/B}=\SUP{a/B}$.\hfill(Compatibility with SU-rank)
		\end{itemize}
	\end{thm}
	\begin{proof}
		Aut-invariance, base closure and local character are immediate from the definition. Forking characterization and compatibility follow from Corollary~\ref{C:gRequivfrk} and \ref{C:gRa=SUa}. 
		
		By taking a suitable finite subset, alg-invariance and monotonicity follow easily from Theorem~\ref{T:welldefined} and Proposition~\ref{P:gRequivstar}. 
		
		For additivity, choose $B_0$ finite such that $ab\indi{P}_{B_0}B$. Then $a\indi{P}_{bB_0}B$ and $b\indi{P}_{B_0}B$, and hence $\R{ab/B}=\R{ab/B_0}=\R{a/bB_0}+\R{b/B_0}=\R{a/bB}+\R{b/B}$. 
		
		For anti-reflexivity, note that by definition $\R{a/aB}=0$. Thus $\R{a/B}=0$ iff $\R{a/B}=\R{a/aB}$. By the forking characterization, this is equivalent to $a\indi{P}_B a$, which holds iff $a\in\aeq{B}$.
	\end{proof}
	\begin{rmk}
		Although our construction of geometric rank was carried out in the setting of pairs of ACF, the definition of geometric rank and its role in witnessing non-forking do not depend on specific features of ACF. The argument only uses elimination of imaginaries and stability, together with the fact that algebraic closure is a model—a property also used in Pillay’s proof of geometric elimination of imaginaries for pairs of ACF. In particular, the same construction applies to beautiful pairs of models of any strongly minimal theory with EI in which the algebraic closure of the empty set is infinite. (Fact~\ref{F:Pillay-gei} continues to hold in that context.)
		
		For Theorem~\ref{T:Pillay-alg}, the analysis of interalgebraicity for imaginaries of Pillay form remains valid in any stable theory. However, in the general stable context, it is not known whether every imaginary is interalgebraic with one of Pillay form.
		
		More generally, it is natural to ask to what extent the assumption that algebraic closure is a model is essential for the above constructions.
	\end{rmk}
	\bibliographystyle{amsalpha}
	\bibliography{aty}	
\end{document}